\documentclass{article}
\usepackage[margin=1.2in]{geometry}
\title{Local-to-global frames and applications to dynamical sampling problem}
\author{A. Aldroubi\thanks{Department of Mathematics, Vanderbilt University, Nashville TN, USA, email: \texttt{akram.aldroubi@vanderbilt.edu}} \and C. Cabrelli\thanks{Departamento de Matem\'atica, Universidad de Buenos Aires, Argentina, email: \texttt{cabrelli@dm.uba.ar}}\and U. Molter\thanks{Departamento de Matem\'atica, Universidad de Buenos Aires, Argentina, email: \texttt{umolter@dm.uba.ar }}\and A. Petrosyan\thanks{Computational and Applied Mathematics Group, Oak Ridge National Laboratory, Oak Ridge TN, USA, email: \texttt{petrosyana@ornl.gov}}}

\usepackage{hyperref}

\date{\vspace{-5ex}}
\usepackage{amsthm}
\usepackage{ dsfont }
\usepackage{amssymb}
\usepackage{enumerate}
\usepackage{graphicx}
\usepackage{float}
\usepackage{bbm}
\usepackage{amsmath}
\usepackage{comment}
\usepackage{hyperref}
\usepackage{listings}
\usepackage{color}
\usepackage[dvipsnames]{xcolor}

\hypersetup{
    colorlinks,
    citecolor=blue,
    filecolor=blue,
    linkcolor=blue,
    urlcolor=blue
}

\newtheorem{theorem}{Theorem}[section]
\newtheorem*{thm}{Theorem}

\newtheorem{proposition}[theorem]{Proposition}

\newtheorem{corollary}[theorem]{Corollary}
\newtheorem{definition}[theorem]{Definition}

\newtheorem{example}{Example}

\newcommand{\HH}{\mathcal{H}}

\newcommand{\R}{\mathbb{R}}
\newcommand{\Z}{\mathbb{Z}}
\newcommand{\N}{\mathbb{N}}

\newcommand{\G}{\mathcal{G}}

\newcommand {\I} {M}
\newcommand {\Id} {\mathds 1}

\newcommand{\supp}{{\rm supp\,}}

\newcommand {\la} {\langle}
\newcommand {\ra} {\rangle}
\newcommand {\GG} {\mathcal{G}}
\begin{document}
\date{ }
\maketitle

\begin{abstract} In this paper we consider systems of vectors in a Hilbert space $\HH$ of the form $\{g_{jk}: j \in J, \, k\in K\}\subset \HH$ where $J$ and $K$ are countable sets of indices. We find conditions under which the local reconstruction properties of such a system extend to  global stable recovery properties on the whole space. As a particular case, we obtain new local-to-global results for  systems of type $\{A^ng\}_{g\in\GG,0\leq n\leq L }$ arising in the  dynamical sampling problem.
\end{abstract}

\section{Introduction}
The notion of frame  was originally introduced by Duffin and Schaffer \cite{DS52} in the context of non-harmonic analysis, in particular exponential systems of the form $\{e^{2\pi i\lambda _ix}\}_{i\in I}$ where $\{\lambda_i \}_{i\in I}\subset \R$ were considered. However, the notion found wide applications in signal processing research and  it has attracted a great deal of interest from the mathematics as well as the engineering communities during the last several decades. 
\begin{definition} A system of vectors $\mathcal{F}=\{f_i\}_{i\in I}\subset \HH$ is said to be a frame for the Hilbert space $\HH$ if there exist two constants $\alpha,\beta>0$ such that \[\alpha\|f\|^2\leq \sum_{i\in I}|\la f,f_i\ra|^2\leq \beta \|f\|^2,\] for all $f\in \HH$. If only the upper bound holds, then the system is called a Bessel system.
\end{definition}

Frames are a generalization of bases. In particular, if $\{f_i\}_{i\in I}$ is a frame for $\HH$, then any vector $f\in \HH$ has the representation $f =\sum_{i\in I}\la f,f_i\ra\tilde f_i $ where $\{\tilde f_i\}_{i\in I}$ is another frame called the canonial  dual frame. However, unlike bases,
the vectors in a frame $\{f_i\}_{i\in I}$ may have linear dependencies thus resulting in a redundant representation (see \cite{christensen2016book} for more details). If the frame system is not a basis, then  a frame $\{f_i\}_{i\in I}$ has many dual frames other than the canonical dual frame. 

In signal processing models,  often the system $\mathcal{F}$ consists of functions, defined on the domain $\Gamma=\Z$ or $\R$. The functions $f_i$ are usually  localized. By that we mean that every  point in the domain is contained  in the supports of only finitely many functions from the system. This may happen, for example,  when $\mathcal{F}$ is generated by the shifts of a finite number of compactly supported functions.  In this setting, we may be able to divide the domain into smaller patches such that  every function $f$ can be recovered on each patch,  from inner products with those functions from $\mathcal{F}$ whose support overlap with that patch, i.e., $f$ can be recovered from local information.
For that reason, we consider systems $\mathcal{F}$ re-indexed to have the form $\{g_{jk}: j \in J, \, k\in K\}\subset \HH$ where the index   $j \in J$ corresponds to the patch $j$ and the index in $k \in K$ corresponds to indexing within  each patch. 
The index sets $I,J,K$ will always be at most countable.
The main interest in this  note is to find conditions that guarantee stable global reconstruction of any  function on the whole domain $\Gamma$, given that the system has   local recovery properties. In that regard, among other results, we prove the following theorem.
\begin {thm} 
Let $\HH$ be a Hilbert space and  $\{P_j\}_{j\in \N}$ be projectors acting on  $\HH$ satisfying $P_rP_j=0$ for $r\ne j$, and also $\sum_jP_j=\mathds 1$. Assume that $\{P_jg_{jk}: k \in K \}$ form a frame for $P_j\HH$ with uniform frame bounds $\alpha, \beta$, and $\|P_rg_{jk}\|\le |c_k(j-r)|$ with  $c_k \in \ell^1(\Z)$ for all $k \in K$. If $\sum\limits_{k \in K }\|c_k\|_1< \alpha$, then $\{g_{jk}\}_{ j\in J, k\in K}$ is a frame for $\HH$  with frame bounds $\alpha-\sum\limits_{k \in K }\|c_k\|_1$ and $\beta+\sum\limits_{k \in K }\|c_k\|_1$. 
\end {thm}

%\textcolor{red}{Are there any other results by other people concerning this local to global thing? Need to mention their work too. }

As a main application of our results, we study the Dynamical Sampling (DS) problem where one wants to recover an unknown signal $f$ from  spatial-temporal samples $\{\la f,A^ng\ra\}_{g\in\GG,0\leq n\leq L }$
where $L<\infty$, and $\GG$ is a collection of functions defined in $\Gamma.$ Here $\Gamma$ is either $\R$  (continuous case) or $\Z$ (discrete case), and $A$
is an operator acting on $L^2(\Gamma)$.  
This leads to the study of frame  properties of the iterative system 
\begin{equation}\label{eq:itertive_sys}
\{A^ng\}_{g\in\GG,0\leq n\leq L }.
\end{equation}

The dynamical sampling problem  was introduced   in \cite{ADK13}, and it is an active area of research in the applied harmonic analysis community, with a range of potential fields of applications \cite{aldroubi2018noise}.
 Early results in DS considered convolution operators on the spaces $l^2(\Z_d),l^2(\Z)$ and $L^2(\R)$ with the  samples taken on a sparse uniform grid \cite{ADK13, ADK15}. 
 In \cite{Aldkr,tang2017system}, authors allowed the operator $A$ to be unknown too. For an operator theoretic approach to the dynamical sampling problem see \cite{aldroubi2017iterative, ACMT14, aldroubihoang2018frames,aldroubipetr2017dynamical,philipp2017bessel}. The  case when $f\in l^2(\N)$,   $A=B^{-1}DB$ and $D=\sum_j\lambda_j P_j$ is an infinite diagonal matrix can be found in  \cite{ACMT14}. For general bounded normal operators see \cite{aldroubipetr2017dynamical, CMPP2018-several,cabrelli2017finite}. The authors in \cite{christensen2017operator, christensenhas2018frame, christensen2018frame,christensen2018dynamical,
christensen2018operator} consider general frames and investigate conditions under which they can be represented as iterated systems arising in dynamical sampling. In \cite{aldroubi2018phaseless} authors study the problem of phaseless recovery from dynamical samples. The scalability properties of the system  arising in DS is investigated in \cite{aceska2017scalability}.

Our main results in this paper are proved for general class of systems and the results are applied to dynamical sampling.  In particular, we consider dynamical systems where the system $\G$ is localized in space (e.g. when it is generated by the shift of a single compactly supported function),  and   $A$ is a compactly supported convolution operator. More specifically, we obtain the following theorem as corollary of the results in Section \ref{sec:main}. 
%\begin {thm}
%Let $a\in \ell_1(\Z)$ be such that for a set $\Omega\subset I=\{0,\dots,N\}$, the system $\{a^k*\delta_i\vert_I: \; i\in \Omega, k=0,\dots,K\}$ is a frame in the space $\ell^2(I)$ with frame bounds $\alpha$ and $\beta$ where by $a^k$ is the $k$ times convolution of $a$. Denote $I_j=I-jN$ and $\Omega_j=\Omega-jN$. If
%\[\gamma=\sum\limits_{i\in\Omega}\sum\limits_{k=0,\dots,K}\sum\limits_{j\in\Z\setminus \{0\}}\|a^k*\delta_i\vert_{I_j}\|_{\ell^2(I_j)}<\alpha\] 
%then the system $\{a^k*\delta_i:\; i\in \cup_{j\in \Z}\Omega_j, k=0,\dots,K\}$ is a frame for $\ell^2(\Z)$ with frame bounds $\alpha-\gamma$ and $\beta+\gamma$.
%\end{thm}
\begin{thm}
Let $a$ be a finite sequence, and consider the convolution operator $Af=a\ast f$ for $f \in \ell^2(\Z)$. Let $I=\{-N,\dots,N\}$ and $I_j=I+N_j\subset \Z$ be such that $\cup_jI_j=\Z$, and $I_j\cap I_k=\emptyset$ for   $\vert j-k\vert >M$. Assume that $G_0=\{\chi_{I_j} a^n(\cdot-l): l\in \Omega\subset I\}$ is a frame for $\ell^2(I)$ ($\chi_{I_j}$ is the characteristic function of $I_j$), and let $S_k$ be the shift  operator by  the integer $k$. Then $\cup_j S_{N_j}G_0$ is a frame for $\ell^2(\N)$. 
\end{thm}

The paper is organized as follows. In Section \ref{sec:main} we prove our main results concerning general systems with local-to-global indexing. In Section \ref{sec:apptoDS}, we apply the results from Section \ref{sec:main} to iterative system of vectors arising from dynamical sampling problem.

\section {Main results for general systems}\label{sec:main}
In this section we  prove our main theorems concerning local-to-global frame properties of the systems $\{g_{jk}: j \in J, \, k\in K\}\subset \HH$  where $J$ and  $K$ are countable (finite or infinite). Here $g_{j,k}$ are some vectors, not necessarily functions on $\R$  or $\Z$. The local patches are replaced with projection operators $\{P_j\}_{j\in J} $ or more generally with operators $\{T_j\}_{j\in J} \subset B(\HH)$. Our first proposition gives a way to construct a global frame from local frames. This is a generalization of Theorem 5.3 from \cite{ACM04}.

\begin {proposition}
\label{OF}
If  $\{T_j\}_{j\in J} \subset B(\HH)$, $T_j$ has close range for each $j$, $A\|f\|^2\le \sum_j\|T_jf\|^2\le  B \|f\|^2$ for all $f  \in \HH$ for some $A,B>0$, and  the set $\{T_jg_{jk}: k \in K\}$ forms a frame for $T_j(\HH)$ with uniform frame bounds $\alpha, \beta$, then $\{T^*_jT_jg_{jk}: j \in J , k \in K\}$  is a frame for $\HH$ with frame bounds $\alpha A, \beta B$.

Conversely,  if $\{T_jg_{jk}: k \in K\}$ forms a frame for $T_j(\HH)$ with uniform frame bounds $\alpha, \beta$, and $\{T^*_jT_jg_{jk}: j \in J , k \in K\}$  is a frame for $\HH$ with frame bounds $\alpha A, \beta B$,   then 
\[
\frac \alpha \beta A\|f\|^2\le \sum_j\|T_jf\|^2\le  
\frac \beta \alpha  B \|f\|^2 \quad \text { for all } f  \in \HH.
\]
\end {proposition}

\begin {proof}
Let $f \in \HH$. Then
 \[
 \alpha \|T_j f\|^2\le \sum_{k\in K} \vert \la T_jf,T_jg_{jk}\ra\vert^2=\sum_{k\in K} \vert \la f,T^*_jT_jg_{jk}\ra\vert^2.
 \] Summing over $j \in J$ we get
\[
\begin{split}
\alpha A \|f\|^2\le \alpha \sum_{j \in J} \|T_jf\|^2\le \sum_{j\in J}\sum_ {k\in K} \vert\la f, T^*_jT_jg_{jk}\ra\vert^2,
\end{split}
\]
which establishes the lower frame bound for the set $\{T^*_jT_jg_{jk}: j \in J , k \in K\}$. A similar calculation yields the upper bound.

To prove the second part,  let $f \in \HH$. Then

\[
\sum_{k\in K} \vert \la f,T^*_jT_jg_{jk}\ra\vert^2=\sum_{k\in K} \vert \la T_jf,T_jg_{jk}\ra\vert^2\le \beta \|T_j f\|^2.
\]
Summing over $j \in J$, we get
\[
\alpha A\|f\|^2\le \sum_{j \in J}\sum_{k\in K} \vert \la f,T^*_jT_jg_{jk}\ra\vert^2\le \beta \sum_{j\in J}\|T_j f\|^2.
\]
Thus, $\frac \alpha {\beta }  A\|f\|^2 \le \sum_{j\in J}\|T_j f\|^2$. Similar calculations yield the upper bound.  
\end{proof}

For the next proposition, we recall the definition of fusion frames (see \cite{CCEL15, CCL12,CKL08,KK11} and the references therein) and completeness for  a  system of orthogonal projectors.
\begin{definition} A system of orthogonal projections
$\{P_j\}_{j\in J}$ is  said to be a fusion frame if there exists positive constants $m,M>0$ such that
\begin{equation} \label {FF}
m\|f\|^2\le \sum_j\|P_jf\|^2\le M\|f\|^2, \quad f \in \HH.
\end{equation}
\end{definition}

\begin{definition}
A system of projections $\{P_j\}_{j\in \N}$ is said to be complete if $P_jf=0$ for all $j$ implies that $f=0$.
\end{definition}

Using the fact that for  orthogonal projector $P_j$, $P_j=P_j^2$ and $P_j=P_j^*$, Proposition \ref {OF}  immediately yields the following corollary:
\begin {proposition}
\label {LGF}
If  $\{P_j\}_{j\in J}$ is  a fusion frame with bounds $A,B$, and  the set $\{P_jg_{jk}: k \in K\}$ form a frame for $P_j(\HH)$ with uniform frame bounds $\alpha, \beta$, then $\{P_jg_{jk}: j \in J , k \in K\}$  is a frame for $\HH$ with frame bounds $\alpha A,\beta B$.

Conversely,  if $\{P_jg_{jk}: k \in K\}$ form a frame for $P_j(\HH)$ with uniform frame bounds $\alpha, \beta$, and $\{P_jg_{jk}: j \in J , k \in K\}$  is a frame for $\HH$ with frame bounds $\alpha A, \beta B$,   then $\{P_j\}_{j\in J}$ is a fusion frame with  bounds $\frac \alpha {\beta }  A, and \frac \beta {\alpha }  B$.
%\[
%m\|f\|^2\le \sum_j\|P_jf\|^2\le M\|f\|^2, \quad f \in \HH
%\]
\end {proposition}
%alpha \|P_jf\|^2\le \sum\limits_{j,k}|\la f,P_jg_{jk}\ra|^2=\sum\limits_{j,k}|\la P_jf,P_jg_{jk}\ra|^2\le 
%\begin {proof}
%Let $f \in \HH$. Assume that s  $\{P_jg_{jk}: k \in K\}$ from a frame for $P_j\HH$ we have
%
%\[
%\alpha \|P_jf\|^2\le \sum_k|\la P_jf,P_jg_{jk}\ra|^2=\sum_k|\la f,P_jg_{jk}\ra|^2\le \beta \|P_jf\|^2,
%\]
%%\[
%%\alpha \|P_jf\|^2\le \sum_k|\la P_jf,P_jg_{jk}\ra|^2=\sum_k|\la f,P_jg_{jk}\ra|^2\le \beta \|P_jf\|^2,
%%\]
%where $\alpha, \beta$ are frame bounds for $\{P_jg_{jk}: k \in K\}$.  If $\{P_j\}_{j\in J}$ is a fusion frame, then summing over $j$ we get that $\{P_jg_{jk}: j \in J, k \in K\}$ from a frame for $\HH$ with frame bound $\alpha A, \beta B$.
%
%Now assume that that $\{P_jg_{jk}: k \in K\}$ from a frame for $P_j\HH$ with uniform frame bounds $\alpha, \beta$, and $\{P_jg_{jk}: j \in J , k \in K\}$  is a frame for $\HH$.  Then 
%
%\[
%\sum_k|\la P_jf,P_jg_{jk}\ra|^2=\sum_k|\la f,P_jg_{jk}\ra|^2\le \beta \|P_jf\|^2.
%\]
%Summing over $j \in J$ and using the fact that $\{P_jg_{jk}: j \in J , k \in K\}$  is a frame for $\HH$with fame bound $a^\prime, B^\prime$, we get
%\[
%A^\prime\|f\|^2\le \sum_j\sum_k|\la P_jf,P_jg_{jk}\ra|^2=\sum_k|\la f,P_jg_{jk}\ra|^2\le \beta \sum_j \|P_jf\|^2.
%\]
%Similarly,
%
%\[
%\alpha\|P_jf\|^2\le \sum_k|\la P_jf,P_jg_{jk}\ra|^2=\sum_k|\la f,P_jg_{jk}\ra|^2.
%\]
%Summing over $j \in J$, we get
%\[
%\alpha\sum_j\|P_jf\|^2\le \sum_k|\la P_jf,P_jg_{jk}\ra|^2=\sum_k|\la f,P_jg_{jk}\ra|^2\le B^\prime \|f\|^2.
%\]
%Thius, $\{P_j\}_{j\in J}$ is a fusion frame with frame bound $A^\prime/\beta$ and $B^\prime/\alpha$.
%\end {proof}

\begin {theorem}\label{thm:completefusion}
Assume that  $\{P_j\}_{j\in \N}$ is such that $P_jP_k=P_kP_j$ for all $j,k\in \N$, and $P_jP_k=0$ for $|j-k|\ge \I$ for some positive integer $\I$. If $\{P_j\}_{j\in \N}$ is complete, then $\{P_j\}_{j\in \N}$ is a fusion frame.  
\end {theorem}

\begin {proof}
The set $\{P_j\}_{j\in \N}$ is the union of $\I$ sets $\{P_j\}_{j\in \N}=\bigcup\limits_{k=1}^\I\{P_{k+\I j}\}_{j\in \N}$. Using the fact that for each $k$,  $\sum\limits_{j \in \N} P_{k+ \I j}$  is an orthogonal projection, we get that for  $f \in \HH$ 

\begin {align*}
\|\sum\limits_{j \in \N} P_{j}f\|^2&= \|\sum\limits_{k=1}^\I \sum\limits_{j \in \N} P_{k+\I j}f\|^2\\
&\le\sum\limits_{k=1}^\I\|\sum\limits_{j \in \N} P_{k+ \I j}f\|^2\\
&\le \I \|f\|^2.
\end{align*}
To prove the lower bound we proceed by induction. We first note that, because $P_1,P_2$ commute, $P_1P_2$ is an orthogonal projector on $P_1(\HH)\cap P_2(\HH)$. Let $Q_1=P_1$, $Q_2=P_2-P_2Q_1$. Then,  $Q_2^*=Q_2$, $Q_2^2=Q_2$,  and $Q_1Q_2=0$. Moreover,  since $P_1Q_2=0$, we get 
\[
\|Q_1f\|^2+\|Q_2f\|^2= \|P_1f\|^2+\|(P_2-P_2P_1)f\|^2=\|P_1f\|^2+\|(\Id-P_1)P_2f\|^2\le \|P_1f\|^2+\|P_2f\|^2.
\]
For $n>2$, let $Q_n=P_n-P_n\sum\limits_{k=1}^{n-1}Q_k$. Then, using the fact that $R=\sum\limits_{k=1}^{n-1}Q_k$ commutes with $P_n$, it is not difficult to verify that $Q^*_n=Q_n$, and $Q_n^2=Q_n$ and that  $Q_nR=0$.  Assume that $Q_jQ_k=0$ for  $j,k<n$. Then, applying $Q_l$ to both side of the equation  $Q_nR=0$, we get that $Q_nQ_l=0$ for all $l<n$. Thus, $Q_jQ_k=0$ for $j\ne k$.  Using induction again, we have 

\[\|Rf\|^2+\|Q_nf\|^2= \sum\limits_{j=1}^{n-1} \|Q_jf\|^2+ \|(\Id-R)P_n f\|^2\le \sum\limits_{j=1}^{n-1} \|P_jf\|^2+ \|P_nf\|^2.\]
Hence 
\[
\sum\limits_{j\in \N}\|Q_jf\|^2\le \sum\limits_{j\in \N}\|P_jf\|^2.
\]
Since $Q_1=P_1$, $Q_n=P_n-P_nR$ for $n\ge 2$, we get that $\{Q_j\}$ is also complete whenever $P_j$ is complete. Hence, 
\[
\|f\|^2= \sum\limits_{j\in \N}\|Q_jf\|^2\le \sum\limits_{j\in \N}\|P_jf\|^2.
\]
\end{proof}
 Combining Proposition \ref {LGF} with Theorem \ref{thm:completefusion}, we  obtain the following corollary. 
\begin {corollary}
\label {CorprojFI}
Assume that  $\{P_j\}_{j\N}$ is complete and is such that $P_jP_k=P_kP_j$ for all $j,k\in \N$, and $P_jP_k=0$ for $|j-k|\ge M$ for some positive integer $M$. Furthermore assume that $\{P_jg_{jk}: k \in K\}$ from a frame for $P_j\HH$ with uniform frame bounds $\alpha, \beta$ and that $P_jg_{jk}=g_{jk}$ for all $j,k$. Then,  the set $\{g_{jk}\}_{ j\in J, k\in K}$ is a frame for $\HH$.
\end {corollary}

In the previous corollary the assumption that $P_jg_{jk}=g_{jk}$ is rather strong and does not always happen in practice. The next  theorem offers a way to overcome this assumption.

\begin {theorem} \label{thm:l1majorized}Let $\{P_j\}_{j\N}$ be orthogonal projectors satisfying $P_rP_j=0$ for $r\ne j$, and $\sum_jP_j=\Id$. Assume that $\{P_jg_{jk}: k \in K \}$ form a frame for $P_j\HH$ with uniform frame bounds $\alpha, \beta$, and $\|P_rg_{jk}\|\le |c_k(j-r)|$ with  $c_k \in \ell^1(\Z)$. If $\sum\limits_{k=1}^K\|c_k\|_1< \alpha$, then $\{g_{jk}\}_{ j\in J, k\in K}$ is a frame for $\HH$  with frame bounds $\alpha-\sum\limits_{k=1}^K\|c_k\|_1$ and $\beta+\sum\limits_{k=1}^K\|c_k\|_1$. 
\end {theorem}
\begin {proof} For the lower bound we have
\begin{align*}  
\Big({\sum\limits_{j,k}\vert\la f,g_{jk}\ra\vert^2}\Big)^{1/2}&=\Big({\sum\limits_{j,k}\big\vert\la \sum_lP_lf,g_{jk}
\ra\big\vert^2}\Big)^{1/2}\\
&=\Big({\sum\limits_{j,k}\big\vert{\la P_jf,g_{jk}\ra+ \la \sum_{l\ne j}P_lf,g_{jk}
\ra}\big\vert^2}\Big)^{1/2}\\
&\ge \Big({\sum\limits_{j,k}\big\vert{\la P_jf,g_{jk}\ra}\big\vert^2}\Big)^{1/2}- \Big({\sum\limits_{j,k}\big\vert{\la \sum_{l\ne j}P_lf,g_{jk}\ra}\big\vert^2}\Big)^{1/2}\\
&\ge \Big({\alpha\sum_j \|P_jf\|^2}\Big)^{1/2}-\Big({\sum\limits_{j,k}\big\vert{ \sum_{l\ne j}\|P_lf\||c_k(l-j)|}\big\vert^2}\Big)^{1/2}\\
&= \alpha^{1/2}\|f\|-\Big({\sum\limits_{j,k}\big\vert{ \sum_{l\ne j}\|P_lf\||c_k(l-j)|}\big\vert^2}\Big)^{1/2}\\
&\ge \alpha^{1/2}\|f\|-\Big({\sum\limits_{k} \sum_{l}\|P_lf\|^2\|c_k\|_1^2}\Big)^{1/2}\\
&= \alpha^{1/2}\|f\|-\Big({\sum\limits_{k} \|c_k\|_1^2}\Big)^{1/2}\|f\|.
\end {align*}
A similar calculation gives the upper bound. 
\end {proof}

\section{Application to DS}\label{sec:apptoDS}
In this section we apply our results from Section \ref{sec:main} to the dynamical sampling setting. As a direct corollary of the Theorem \ref{thm:l1majorized}, we get the following theorem.
\begin {theorem}\label{cor:convframe}
Let $a\in \ell_1(\Z)$ be such that for a set $\Omega\subset I=\{0,\dots,N\}$, the system $\{a^k*\delta_i\vert_I: \; i\in \Omega, k=0,\dots,K\}$ is a frame in the space $\ell^2(I)$ with frame bounds $\alpha$ and $\beta$ where  $a^k$ denotes the $k$ times convolution of $a$. Let $I_j=I-jN$ and $\Omega_j=\Omega-jN$. If
\[\gamma=\sum\limits_{i\in\Omega}\sum\limits_{k=0,\dots,K}\sum\limits_{j\in\Z\setminus \{0\}}\|a^k*\delta_i\vert_{I_j}\|_{\ell^2(I_j)}<\alpha\] 
then the system $\{a^k*\delta_i:\; i\in \cup_{j\in \Z}\Omega_j, k=0,\dots,K\}$ is a frame for $\ell^2(\Z)$ with frame bounds $\alpha-\gamma$ and $\beta+\gamma$.
\end{theorem}
To insure the assumptions in the above theorem are realistic, below we construct an  example where they actually hold.
\begin{example}
Let $I=\{0,1,2\}$, $\Omega=\{0\}$ and 
\[a(i)=
\begin{cases}
1,&i=0\\
\tau,& i=1\\
\tau^2,& i=2\\
0,&\text{otherwise}
\end{cases}
\]
where $\tau>0$. Then notice that 
\[a*a(i)=
\begin{cases}
1,&i=0\\
2\tau,& i=1\\
3\tau^2,& i=2\\
2\tau^3,& i=3\\
\tau^4,& i=4\\
0,&\text{otherwise}
\end{cases}.
\]
Hence 
\[
\{a^k*\delta_0\vert_I:k=0,1,2\}=\{(1,0,0),(1,\tau,\tau^2),(1,2\tau,3\tau^2)\}
\]
which are linearly independent for $\tau>0$ since for the matrix
\[
\Phi=\left(\begin{matrix}
1&0&0\\
1 &\tau &\tau^2\\
1 &2\tau&3\tau^2\\
\end{matrix}\right)
\]
$\det(\Phi)=5\tau^3$
and hence are a basis in $\ell^2(I)$.

The quantity $\gamma$ in Theorem \ref{cor:convframe} is equal to $2\tau^3+\tau^4$.
For sufficiently small values of $\tau$, it can be checked that the lower frame bound of the system, which is the smallest singular value of $\Phi$ is larger than $\gamma$ and thus the conditions of Theorem  \ref{cor:convframe} are satisfied.  For the value of $\tau=0.1$, the lower frame bound is approximately $0.0040$ whereas $\gamma=0.0021$.
\end{example}

The next two propositions concern frames from dynamical sampling from  diagonal and convolutional operators, and their proofs  follow from Corollary \ref {CorprojFI}.
\begin {proposition}\label{prop:diag}
Let $D$ be an infinite  diagonal matrix acting on $\ell^2(\N)$. Assume that $D
\in B(\ell^2)$. Let $\{I_j\}_{j\in \N}\subset \N$ be such that, $\vert I_j\vert\le L$, $\cup_jI_j=\N$, and $I_j\cap I_k=\emptyset$ for  $\vert j-k\vert>M$. Let $D_j$ be the submatrix of $D$ indexed by $I_j$, and let $G_j \subset \ell^2(\N)$ be a set vectors such that, for any $g \in G_j$, $\supp g\subset I_j$ and $\{D_j^ng\}_{g \in G_j, 0\le n< L}$  is a frame for $\ell^2(I_j)$ with frame bounds $\alpha, \beta$ independent of $j$. Then the set $\{D^ng: 0\le n< L, g \in \cup_{j\in\N} G_j\}$ is a frame for $\ell^2(
\N)$.
\end {proposition}

As a specific instance let us consider the next two examples. 
\begin{example}
Assume that  the entries of the diagonal matrix $D$ are periodic on the diagonal, i.e., $ D_{ii}=f(i)$ where $f(i+p)=f(i)$ for $i \in \N$. This case reduces, each submatrix $D_j$ of $D$ in the proposition to be  one of finitely many possible matrices. For each of these matrices, we construct corresponding $G_j$. Since the construction uses finitely many matrices and finitely many associated vectors, we can find frame bounds $\alpha, \beta$ that are  independent of $j$ and the conditions in Proposition \ref{prop:diag} will be satisfied. 
\end{example}

%%{\color{red}
%%\begin{example}
%%We now want to construct functions $g_{j,k}$ such that they are local frames for the intervals $I_j$, by iterations of our diagonal operator $D = \sum_t \alpha_t P_t$. The idea behind this is to think that we have a bunch of $\ell_1$ functions essentially supported on $I$, say $g_0, \dots, g_\alpha$ and each set $g_{j,k}$ is the translations onto $I_j$ of these functions and its iterations by the diagonal matrix $D$.
%%
%%We consider $t_j = 2Nj$ and $I_j = [2Nj-N, 2Nj+N]$.
%%We have $A:\ell_2 \longrightarrow \ell_2$, $A^* = A$ and $U$ is a unitary matrix such that $UAU^* = D$, and $D = \sum_{t\in \Z} \lambda_t P_t$, with $\lambda_t \not= \lambda_{t^{\prime}}$.
%%
%%Let $\tau_j = \{g_{j,1} \dots, g_{j,\alpha_j}\}$ be a set of $\alpha_j$ functions essentially supported on $I_j$ such that 
%%$$ {\rm span} P_t(\{g \in \tau_j\}) = E_t \cap I_j, t \in  \Z.$$
%%(Note that $E_t \cap I_j$ will be empty most of the time!) Then we know that there exists $0<L< |I_j|$ such that
%%$$\{D^k g: 0\leq k \leq L, g \in \tau_j \} ,$$
%%% \{g_{j,1} \dots, g_{j,\alpha_j}, Dg_{j,1} \dots, Dg_{j,\alpha_j}, \dots, D^Lg_{j,1} \dots, D^Lg_{j,\alpha_j} \} $$
%%is a frame of $\ell_2(I_j)$. The $g_j$ are columns of $U$ and should be in $\ell_1$. Note that if $g$is an $\ell_1$ function, there is an $N$ such that outside of $[-N,N]$ the function is arbitrarily small.
%%\end{example}
%%
%%}

The next proposition concerns convolutional operators which covers a large class of operators of practical interest  like the diffusion operator on $\ell^2(\Z)$.
\begin {proposition}
Let $a$ be a finite sequence, and consider the convolution operator $Af=a\ast f$ for $f \in \ell^2(\Z)$. Let $I=\{-N,\dots,N\}$, $\Omega \subset I$, and $I_j=I+N_j\subset \Z$ be such that $\cup_jI_j=\Z$, and $I_j\cap I_k=\emptyset$ for   $\vert j-k\vert >M$. Assume that $G_0=\{\chi_{I_j} a^n(\cdot-l): l\in \Omega\}$  is a frame for $
\ell ^2(I)$  ($\chi_{I_j}$ is the characteristic function of $I_j$), and let $S_k$ be the shift  operator by  the integer $k$. Then $\cup_j S_{N_j}G_0$   is a frame for $\ell^2 (\N)$. 
\end {proposition}

%\section*{Questions}
%\begin{enumerate}
%\item How are $N$ and $m$ related?
%\item If $\tau := \{g_1, \dots, g_\alpha\}$ is a frame sequence with bounds $m_j$ and $M_j$ what are the framebounds of $\{ D^kg: k=0,1, g\in \tau\}$.?
%\item if $D \in \CC^{N\times N}$, if we iterate a single vector, we need $N$ iterations - if we take 2, can we possibly only need $N/2$ iterations?
%{\bf What is the relation between the values of $\lambda_i$ and the vectors (functions) $g_j$?}
%\end{enumerate}

\section*{Acknowledgments}
 The research of A. Aldroubi is supported in part by NSF grant DMS- 1322099. C. Cabrelli and U. Molter were partially supported by grants: UBACyT 20020170100430BA, PICT 2014-1480 (ANPCyT) and CONICET PIP 11220110101018. A. Petrosyan acknowledges support   by the Oak Ridge National Laboratory, which is operated by UT-Battelle, LLC., for the U.S. Department of Energy under Contract DE-AC05-00OR22725. 

%%\bibliography{refers,Akram_refs}{}
%%\bibliographystyle{plain}

\end{document}